\documentclass{IEEEtaes}

\usepackage[colorlinks,urlcolor=blue,linkcolor=blue,citecolor=blue]{hyperref}

\usepackage{color,array,amsthm}
\definecolor{change}{rgb}{0.4,0.4,0.4}

\usepackage{graphicx}

\usepackage{amsmath}

\usepackage{amsthm}

\usepackage{amsfonts}

\usepackage{amsmath,lipsum}

\usepackage{cuted}

\usepackage{multirow}

\usepackage{subfigure}

\usepackage[T1]{fontenc}
\jvol{XX}
\jnum{XX}
\jmonth{XXXXX}
\paper{1234567}
\pubyear{2020}
\doiinfo{TAES.2020.Doi Number}

\newtheorem{theorem}{Theorem}

\newtheorem{assumption}{Assumption}

\DeclareMathOperator{\sat}{sat}
\DeclareMathOperator{\sgn}{sgn}

\begin{document}

\title{Maneuvering tracking algorithm for reentry vehicles with guaranteed prescribed performance}

\author{ZONGYI GUO}
\member{Member, IEEE}

\author{XIYU GU}

\author{YONGLIN HAN}

\author{JIANGUO GUO}
\affil{Northwestern Polytechnical University Xi’an, China}

\author{THOMAS BERGER}
\affil{Institut für Mathematik, Universität Paderborn, Germany}



\authoraddress{Z. Guo, X. Gu, Y. Han, and J. Guo are with the Institute of Precision Guidance, and Control, Northwestern Polytechnical University, Xi'an 710072,
	China, (e-mail:
	guozongyi@nwpu.edu.cn;
	\textcolor{black}{guxiyu@mail.nwpu.edu.cn;} hanyonglin@mail.nwpu.edu.cn;
	guojianguo@nwpu.edu.cn).
	T.\ Berger is with the Universität Paderborn, Institut für Mathematik, Warburger Str. 100, 33098 Paderborn, Germany, (thomas.berger@math.upb.de). (Corresponding author: Zongyi Guo.)}

\editor{This work was supported by the National Natural Science Foundation of China under Grant 92271109 and Grant 52272404 and the Deutsche Forschungsgemeinschaft (DFG, German Research Foundation) under Project-IDs 362536361 and 471539468.}

\markboth{AUTHOR ET AL.}{SHORT ARTICLE TITLE}
\maketitle
\begin{abstract}
	This paper presents a prescribed performance-based tracking control strategy for the atmospheric reentry flight of space vehicles subject to rapid maneuvers during flight mission.
A time-triggered non-monotonic performance funnel is proposed with the aim of constraints violation avoidance in the case of sudden changes of the reference trajectory.
		Compared with traditional prescribed performance control methods, the novel funnel boundary is adaptive with respect to the reference path and is capable of achieving stability under disturbances.
		A recursive control structure is introduced which does not require any knowledge of specific system parameters. By a stability analysis we show that the tracking error evolves within the prescribed error margin under a condition which represents a trade-off between the reference signal and the performance funnel.  The effectiveness of the proposed control scheme is verified by simulations.
	
	\end{abstract}
\section{\textcolor{black}{INTRODUCTION}}

For the cost efficiency of space missions it is imperative that spacecrafts are able to return to earth through its atmosphere, following a prescribed trajectory. Such atmospheric reentry problems are a main focus of the aerospace industry and have received a large amount of attention during the last decade. Trajectory tracking strategies for reentry vehicles (RVs), characterized by high flight velocity, short response time and large envelope advantages, have long been considered as a hot research area due to its extensive applications in engineering \cite{Han Tuo 2021 TII}-\cite{Mehra 1971 ITAC}. Maneuvering  flight has stimulated extensive research in the areas of  evasion, pursuit and obstacle avoidance  for missions achievement \cite{Peng 2019 TCyber}-\cite{H.I. Yang 2004 CD}. Commonly, the main objective of maneuvering flight control  is stability and robustness of the system and to provide the stabilization capabilities in RV tracking either on-line  or for off-line planned reference trajectories. Some widespread control approaches, including PID \cite{P.E. Pounds 2012 Auton}, sliding mode control \cite{Elmokadem 2016 ND}, backstepping control \cite{Cho 2020 Ocean}, adaptive control \cite{Beikzadeh 2018 Franklin}-\cite{Z.Zuo 2014 Electron} and intelligent algorithms \cite{Mu 2016 TNNLS}-\cite{Peng 2017 TSMC}, are  the popular choice owing to their simplicity and effectiveness in RV tracking problems. However, apart from  the steady state characteristics, also the transient behavior must be taken into account for a successful mission. As RVs travel at very high velocities, the aforementioned  control methods, which only focus on the steady state behavior, are not applicable in this context. Although those results demonstrated that the tracking objective can be successfully accomplished, it remains an open issue how to guarantee its high speed convergence, minimum accuracy and small overshoot.

To resolve these drawbacks, control algorithms for constraining the transient performance are flourishing during the past few decades, and two different approaches have been developed. Prescribed performance control (PPC) has been proposed in \cite{Bechlioulis 2008 TAC}-\cite{Bechlioulis 2014 Auto} and is regarded as a representative nowadays. It relies on an error transformation, which is designed to transform the original output error restrictions into an equivalent interval one. Since its universal control structure, PPC has been thoroughly investigated in combination with unconstrained control methods like backstepping and sliding mode control. Funnel control (FC) is the second control mechanism for guaranteeing a prescribed performance of the tracking error
\cite{Ilchmann A 2002}-\cite{Berger 2022 auto} by introducing a time-varying high-gain feedback in the control law. If the error tends towards the funnel boundary, the gain increases so that the error is kept inside the performance funnel.
Both PPC and FC have already been investigated for the control algorithm design for RVs with a focus on the transient performance \cite{XiyuGu 2022 ISA}-\cite{Buxiangwei 2016 Franklin}.
Notice that almost all funnel boundaries selected in works on funnel control are monotonically decreasing functions, although the theory guarantees the stability of the closed-loop system for a vast variety of non-monotonic funnel boundaries. It turns out that for flight maneuvering, non-monotonic funnel boundaries are more suitable. Sudden flight maneuvers may lead to a drastic increase of the control effort or even drive the tracking error across the funnel boundary, resulting in closed-loop system instability.

Considering the various demands of trajectory tracking for RVs during different phases, a control law which adapts itself to the prescribed boundary function is required to guarantee the transient  behavior in flight maneuvering.
Therefore, a time triggered non-monotonic funnel boundary is proposed in this paper. Using a priori information of the reference trajectory, we design a boundary function which is widened during critical phases, e.g.\ in the case of sudden course corrections. In this way, peaks in the control input signal are avoided. 
Additionally,  disturbances (such as noises, uncertainties or unmodeled dynamics) are taken into account, as they might have a detrimental effect on the system's performance. A significant challenge in stability analysis is imposed by the dynamics of the RV, where the flight altitude (which is the system output) is influenced by the deflection angle only in a saturated way (via a $\sin$ function). Therefore, arbitrary instantaneous changes of the altitude are not possible. Any prescribed trajectory can only be tracked up to a certain accuracy, i.e., there is a trade-off between the derivative of the reference and the funnel boundary function. To the best of our knowledge, this is the first work where such a trade-off is found for RV tracking problems with guaranteed prescribed performance.

Focusing on the issues mentioned above, the main contributions of this paper are summarized as follows.
\begin{itemize}
   \item A recursive funnel control structure with low complexity is adopted to avoid the requirement of a priori knowledge of system parameters.
	\item We design a  robust funnel control law with time triggered non-monotonic funnel boundary for flight maneuvering with guaranteed transient performance under disturbances.
	\item We derive a condition representing a trade-off between the reference trajectory and the funnel boundaries, under which the RV is amenable to the proposed funnel control law.	
\end{itemize}

The remainder of this paper is organized as follows. The dynamic model of the RV in yaw channel is presented in  Section \ref{section PROBLEM FORMULATION}, together with the proposed funnel boundary and the control objective.  In Section \ref{FUNNEL CONTROLLER DESIGN} we state the funnel-based control law design and provide the stability analysis.
Simulation results are given in Section \ref{section simulation} and Section \ref{section conclusion} finally concludes this article.

\section{PROBLEM FORMULATION}\label{section PROBLEM FORMULATION}
\subsection{RV Dynamics}
\textcolor{black}{Considering horizontal lateral maneuvers with constant velocity, the simplified model of a RV in yaw channel is established in \cite{Chai RQ Acta 2020,Guanjie Hu 2022 TII} and is of the form}
\begin{align}\label{dzh}
{{\dot z}_h}\left( t \right) &=  - V\sin \left( {{\psi _V}\left( t \right)} \right)+\Delta_0(t)\\\label{dpsi_V}
{{\dot \psi }_V}\left( t \right) &=  - \frac{1}{{mV}}Z\left( t \right)+\Delta_1(t)\\ \label{dpsi}
\dot \psi \left( t \right) &= {\omega _y}\left( t \right)+\Delta_2(t)\\ \label{dw}
{{\dot \omega }_y}\left( t \right) &= \frac{{{M_y}\left( t \right)}}{{{J_y}}}+\Delta_3(t)\\\label{beta}
\beta \left( t \right) &= \psi \left( t \right) - {\psi _V}\left( t \right)
\end{align}
where $z_h$ is the flight altitude, $m$ is the mass, $V$ is the flight speed, $\psi_V$, $\psi$, $\beta$ represent the deflection angle, yaw angle and sideslip angle, $\omega_y$ is the yaw rate, $J_y$ denotes the yaw rotational inertia and \textcolor{black}{$\Delta_i$ $(i=0,1,2,3)$ are bounded disturbances}. The functions $Z$ and $M_y$ are the aerodynamic force and moment in yaw channel, expressed by $Z(t) = \bar qS( {c_z^\alpha \alpha  + c_z^\beta \beta(t)   + c_z^0} )$, ${M_y}(t) = \bar qSl( {c_M^\alpha \alpha  + c_M^\beta \beta(t)  + c_M^{{\delta _y}}{\delta _y}(t) + c_M^0} )$, where $\alpha$ is the angle of attack, $\delta_y$ represents the rudder angle, $\bar q$ is the dynamic pressure, $S$ and $l$ are the reference area and aerodynamic chord of the RV, and $c_z^i$ and $c_M^j$ $(i=\alpha,\beta,0,j=\alpha,\beta,\delta_y,0)$ are the aerodynamic coefficients for force and moment, respectively. The rudder angle $\delta_y$ can be manipulated and serves as the control input.

We introduce the variables $y_0(t)=z_h(t)$, $y_1(t)=\psi_V(t)$, $y_2(t)=\beta(t)=\psi(t)-\psi_V(t)$, $y_3(t)=\omega_y(t)$ and the control input $u(t)=\delta_y(t)$ to rewrite the dynamic model~\eqref{dzh}--\eqref{beta} in the form
\begin{align}\label{dy0}
{{\dot y}_0}\left( t \right) &=  - V\sin \left( {{y_1}\left( t \right)} \right) + {\Delta _0}(t)\\\label{dy1}
{{\dot y}_1}\left( t \right) &=  - {c_1} - {c_2}{y_2}\left( t \right) + {\Delta _1}(t)\\
{{\dot y}_2}\left( t \right) &= {y_3}\left( t \right) + {c_1} + {c_2}{y_2}\left( t \right) + {\Delta _2}(t) - {\Delta _1}(t)\\\label{dy3}
{{\dot y}_3}\left( t \right) &= {c_3} + {c_4}{y_2}\left( t \right) + {c_5}u\left( t \right) + {\Delta _3}(t)
\end{align}
with the constants ${c_1} = \frac{1}{{mV}}\bar qS\left( {c_z^\alpha \alpha  + c_z^0} \right)$, ${c_2} = \frac{1}{{mV}}\bar qSc_z^\beta$, ${c_3} = \frac{1}{{{J_y}}}\bar qSl\left( {c_M^\alpha \alpha  + c_M^0} \right)$, ${c_4} = \frac{1}{{{J_y}}}\bar qSc_M^\beta$, ${c_5} = \frac{1}{{{J_y}}}\bar qSc_M^{{\delta _y}}$.

	It can be seen from the dynamics (\ref{dzh}) that the flight altitude $z_h$ is influenced only by the deflection angle $\psi_V$ and that this influence is saturated by the $\sin$ function. Therefore, it is clear that it is impossible to achieve tracking of arbitrary reference signals with arbitrary prescribed performance. There must be a trade-off between the reference signal (${{\dot z}_{h_{ref}}}$) and the funnel boundary. This trade-off is formulated as condition (\ref{inequality-step 1}) in Theorem~\ref{theorem}.

\begin{figure}[h]
	\centering
	\includegraphics[width =0.9\columnwidth]{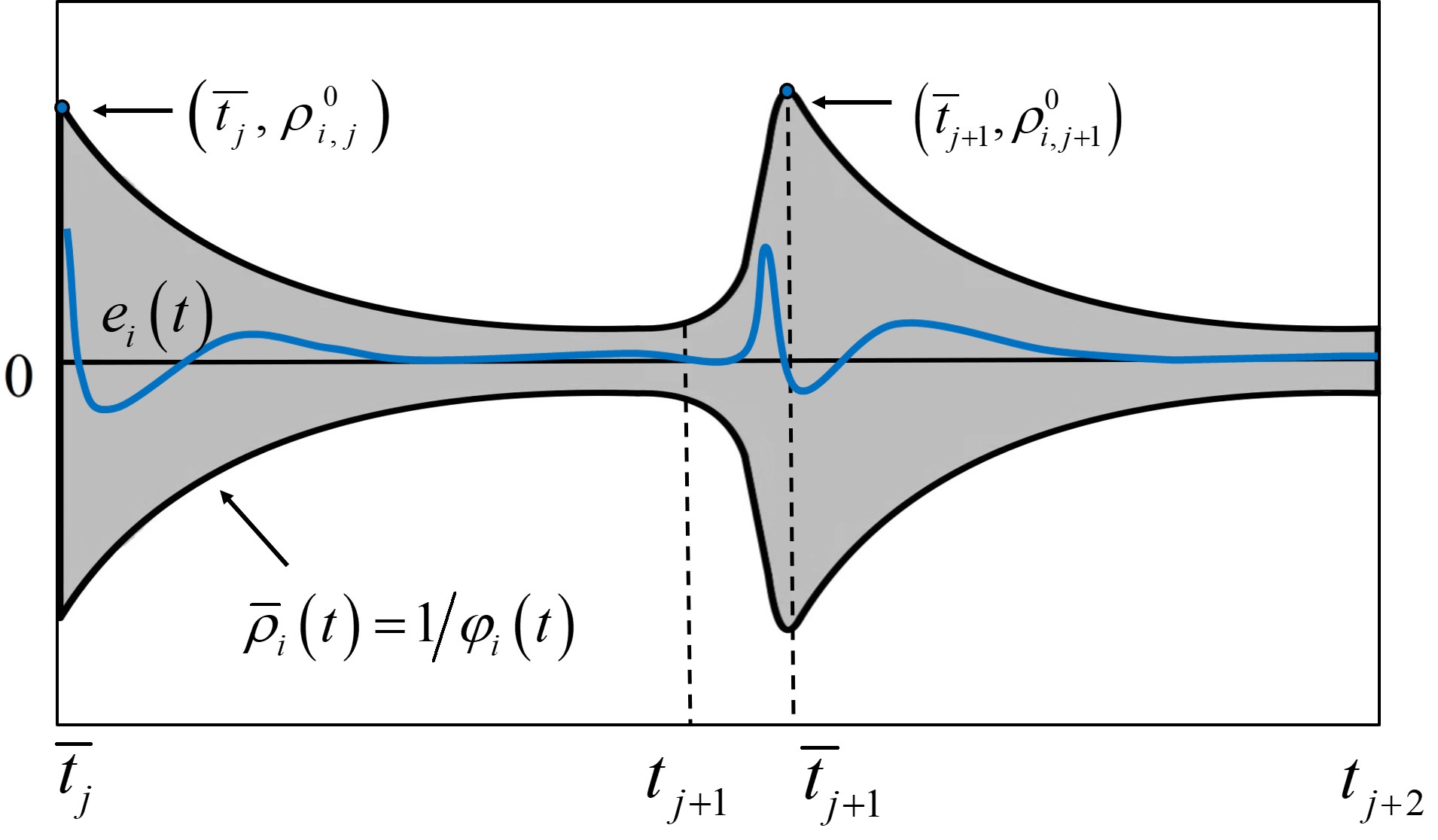}
	\caption{\textcolor{black}{Tracking error evolution in the funnel ${\Gamma_{\varphi_i}}(i=0,1,2,3)$.}}
	\label{fig_funnel}
\end{figure}

\subsection{Funnel Boundary}
We define functions $\varphi _i$ $(i=0,1,2,3)$ as the reciprocal of the funnel boundary ${\bar{\rho}_i}\left( t \right)$, which describe the performance funnel ${\Gamma_{\varphi_i}}$ (cf.~\cite{Ilchmann A 2002}) as
\begin{equation}\label{funnel definition}
\begin{aligned}
{\Gamma _{\varphi_i}}: &= \{ \left( {t,{e_i}} \right) \in {\mathbb{R}_{ \ge 0}} \times \mathbb{R}\mid {\varphi_i}\left( t \right)|{e_i}| < 1\}.
\end{aligned}
\end{equation}
The \textcolor{black}{functions $\varphi_i$ are continuously differentiable, bounded with bounded derivatives,} and satisfy $\varphi _i(t)> 0$ for all $t\ge 0$ and $\liminf_{t\to\infty} \varphi _i(t)>0$.
Fig. \ref{fig_funnel} displays the reciprocal $\bar \rho_i(t)=1/\varphi_i(t)$.

In order to avoid peaks in the control input signal due to a strongly varying reference trajectory, we employ a non-monotonic funnel boundary defined as
\begin{align}\nonumber\label{novel funnel}
&{1 \mathord{\left/
		{\vphantom {1 {{\varphi _i}\left( t \right)}}} \right.
		\kern-\nulldelimiterspace} {{\varphi _i}\left( t \right)}} ={{\bar \rho }_i}\left( t \right)=\\
&\left\{ {\begin{array}{*{20}{c}}
	{{{\bar \rho }_{i,0}}\left( t \right)}&{0 \le t < {t_1}}\\
	\begin{array}{l}
	{a_{i,0}}{\left( {t - {t_1}} \right)^3} + {b_{i,0}}{\left( {t - {t_1}} \right)^2}\\
	\;\;\;\;\;\;\;\;\;\;\;\;+ {c_{i,0}}\left( {t - {t_1}} \right) + {d_{i,0}}
	\end{array}&{{t_1} \le t < {{\bar t}_1}}\\
	{{{\bar \rho }_{i,1}}\left( t \right)}&{{{\bar t}_1} \le t < {t_2}}\\
	\vdots & \vdots \\
	{\begin{array}{*{20}{l}}
		\begin{array}{l}
		{a_{i,j - 1}}{\left( {t - {t_j}} \right)^3} + {b_{i,j - 1}}{\left( {t - {t_j}} \right)^2}\\
		\;\;\;\;\;\;\;\;\;\;\;\; + {c_{i,j - 1}}\left( {t - {t_j}} \right) + {d_{i,j - 1}}
		\end{array}
		\end{array}}&{{t_j} \le t < {{\bar t}_j}}\\
	{{{\bar \rho }_{i,j}}\left( t \right)}&{{{\bar t}_j} \le t < {t_{j + 1}}}\\
	\vdots & \vdots \\
	\begin{array}{l}
	{a_{i,p - 1}}{\left( {t - {t_p}} \right)^3} + {b_{i,p - 1}}{\left( {t - {t_p}} \right)^2}\\
	\;\;\;\;\;\;\;\;\;\;\;\; + {c_{i,p - 1}}\left( {t - {t_p}} \right) + {d_{i,p - 1}}
	\end{array}&{{t_p} \le t < {{\bar t}_p}}\\
	{{{\bar \rho }_{i,p}}\left( t \right)}&{{{\bar t}_p} \le t } 
	\end{array}} \right.
\end{align}
where $t_{j}$ and $\bar t_j$ $(j=1,...,p)$ are the triggered time and initial time points of every phase after maneuvering, ${{\bar t}_j} - {t_j}$ is the time range for maneuvering and $p$ represents the number of triggered times. The polynomials in each interval of the form $[t_j,\bar t_j]$ are chosen based on the current maneuver encoded in the reference trajectory, and they ensure a widening of the funnel boundary. The functions ${{{\bar \rho }_{i,j}}\left( t \right)}\;(i=0,1,2,3,j=1,2,...,p)$ are of the form
\begin{align}\label{traditional boundary}
{{\bar \rho }_{i,j}}\left( t \right) = \left( {\rho _{i,j}^0 - \rho _{i,j}^\infty } \right){e^{ - {l_{i,j}}t}} + \rho _{i,j}^\infty
\end{align}
with initial funnel width ${{\rho^0 _{i,j}}}>0$, required minimum exponential convergence rate $l_{i,j}>0$ and the maximum steady state error $\rho _{i,j}^\infty>0$, respectively. In order to guarantee that $\bar \rho_i$ is continuously differentiable, the parameters $a_{i,j}, b_{i,j}, c_{i,j}, d_{i,j}$ are chosen such that $\mathop {\lim }\limits_{t \to t_j^ - } {{\bar \rho }_{i}}\left( t \right) = \mathop {\lim }\limits_{t \to t_j^ + } {{\bar \rho }_{i}}\left( t \right)$, $\mathop {\lim }\limits_{t \to \bar t_j^ - } {{\bar \rho }_{i}}\left( t \right) = \mathop {\lim }\limits_{t \to \bar t_j^ + } {{\bar \rho }_{i}}\left( t \right)$, $\mathop {\lim }\limits_{t \to t_j^ - } {{\dot {\bar \rho} }_{i}}\left( t \right) = \mathop {\lim }\limits_{t \to t_j^ + } {{\dot {\bar \rho} }_{i}}\left( t \right)$, $\mathop {\lim }\limits_{t \to \bar t_j^ - } {{\dot {\bar \rho} }_{i}}\left( t \right) = \mathop {\lim }\limits_{t \to \bar t_j^ + } {{\dot {\bar \rho} }_{i}}\left( t \right)$ for $j=1,\ldots,p$.

The proposed funnel boundary is displayed in Fig.~\ref{fig_funnel} and we stress that it is different from the monotonically decreasing boundary functions of the form (\ref{traditional boundary}) widely used in \cite{Bechlioulis 2014 Auto}, \cite{Berger 2018 auto}, \cite{Berger 2022 auto}.
Instead, it is a time triggered mechanism with trigger time points $t_j$, chosen in accordance with the reference trajectory, so that the proposed funnel (\ref{novel funnel}) adapts itself and is suitable for the time-varying maneuvering command of RVs. For $t>t_p$ the novel funnel boundary converges to a neighbourhood of the origin, satisfying $\mathop {\lim }\limits_{t \to \infty } {{\bar \rho }_{i}}\left( t \right) = \rho _{i,p}^\infty  > 0$.

\subsection{\textcolor{black}{Control Objective}}
The control objective is to design an output derivative feedback such that for any  sufficiently smooth reference trajectory ${z_{h_{ref}}}$, any initial values and under the influence of disturbances, the tracking error $z_h - z_{h_{ref}}$ evolves within a prescribed performance funnel $\Gamma_{\varphi_0}$ as in (\ref{funnel definition}) and hence exhibits the desired transient and steady behavior. Furthermore, all signals
$u$,  $z_h$, $\psi_V$, $\psi$, $\beta$ and $\omega_y$ in the closed-loop system should remain bounded.

\section{FUNNEL CONTROLLER DESIGN}\label{FUNNEL CONTROLLER DESIGN}

\subsection{\textcolor{black}{Funnel-based control law design}}

Before we define the control law we introduce the following assumptions.
\begin{assumption}\label{assumption_referene}
The reference trajectory ${z_{h_{ref}}}$ is known, it is four times continuously differentiable and its first four derivatives are bounded.
\end{assumption}
\begin{assumption}\label{assumption_disturbance}
	\textcolor{black}{The disturbances $\Delta_i$ are  measurable and essentially bounded with $\|\Delta_i\|_\infty \le D_i$ for known constants $D_i\ge0$ $(i=0,1,2,3)$.}
\end{assumption}

\textcolor{black}{Assumptions \ref{assumption_referene} and \ref{assumption_disturbance} are reasonable and frequently used in the literature.
The disturbances $\Delta_i$ $(i =0, 1, 2,3) $ involved in (\ref{dzh})-(\ref{beta}) account for uncertainties in the aerodynamic coefficients, noises and external disturbances, which are usually bounded throughout the flight process.}

We define the tracking error as
\[
    e_0(t)=y_0(t)-{z_{h_{ref}}(t)}=z_h(t)-{z_{h_{ref}}(t)}
\]
and introduce the following recursive structure
\begin{equation}
\begin{aligned}\label{e_i}
{e_i}\left( t \right) = {y_i}\left( t \right) - z_{{h_{ref}}}^{\left( i \right)}\left( t \right) + {k_{i-1}}\left( t \right){\varpi _{i-1}}\left( t \right),\ i=1,2,3
\end{aligned}
\end{equation}
where ${k_i}\left( t \right) = \frac{1}{{1 - \varpi _i^2\left( t \right)}}$ and ${\varpi _i}\left( t \right) = {\varphi _i}\left( t \right){e_i}\left( t \right)$ for $\varphi_i$ as in~\eqref{novel funnel}.

Then the funnel-based control law is given by
\begin{align}\label{novel funnel controller}
u\left( t \right) =  - k_3\left( t \right){{e_3}\left( t \right)} =  - \frac{{{e_3}\left( t \right)}}{{1 - \varphi _3^2\left( t \right)e_3^2\left( t \right)}}
\end{align}
and the block diagram of the proposed control scheme is shown in Fig.~\ref{liuchengtu}.
\begin{figure}
	\centering
	\includegraphics[width =1\columnwidth]{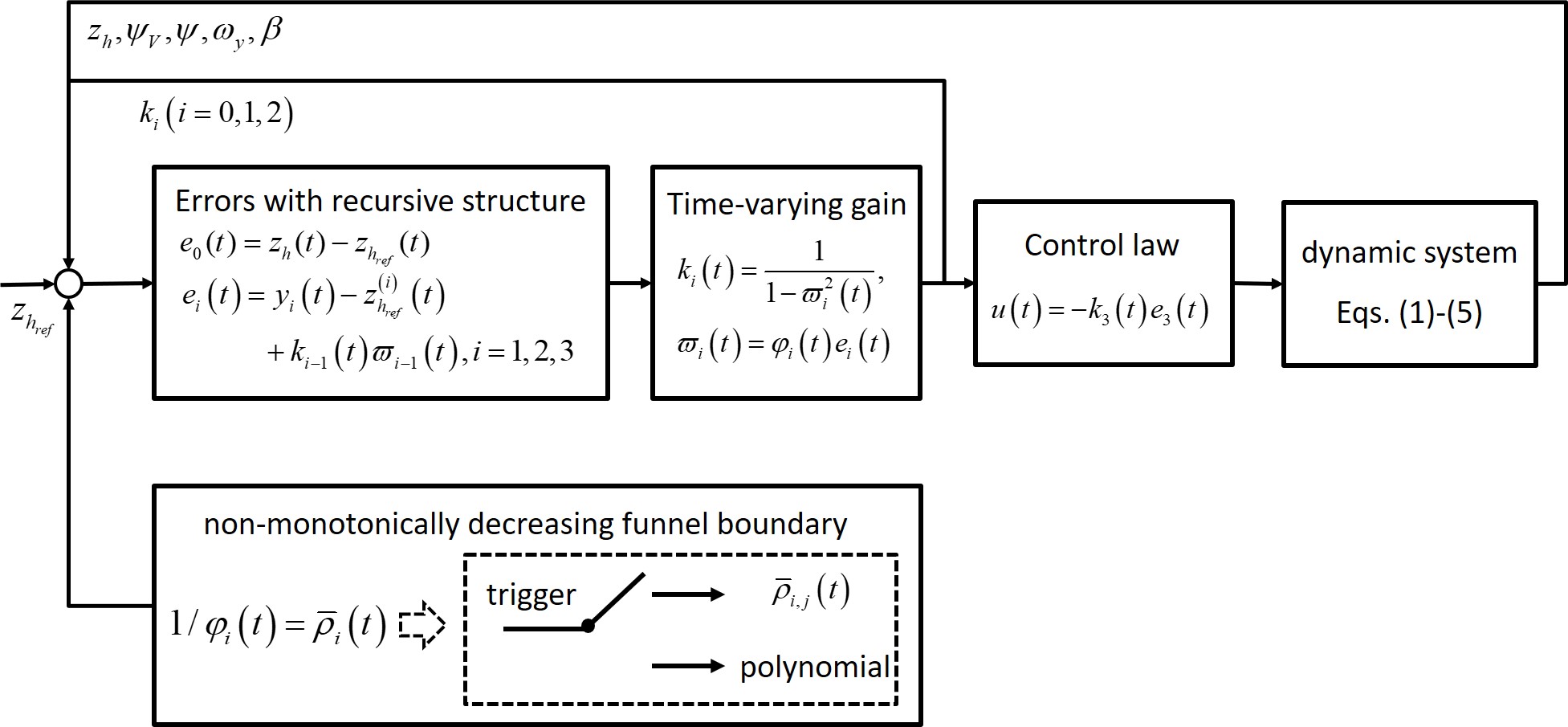}
	\caption{\textcolor{black}{Funnel control structure for RV tracking issue.}}
	\label{liuchengtu}
\end{figure}
In the sequel we investigate existence of solutions of the initial value problem resulting from the application of the funnel controller~\eqref{novel funnel controller} to the RV with dynamics~(\ref{dzh})-(\ref{beta}). By a solution of~(\ref{dzh})-(\ref{beta}),~\eqref{novel funnel controller} we mean a function $(z_h,\psi_V,\psi,\omega_y):[0,t_f)\to\mathbb{R}^4$, $t_f\in(0,\infty]$,
which is locally absolutely continuous and satisfies the initial conditions as well as the differential equations~(\ref{dzh})-(\ref{beta}) for almost all $t\in [0,t_f)$.  A solution is called maximal, if it has no right extension that is also a solution.

\subsection{Stability Analysis}\label{Stability Analysis}
In this part,
we present the stability analysis of the proposed control law.
\begin{theorem}\label{theorem}
    Consider a RV with dynamics~(\ref{dzh})-(\ref{beta}), satisfying Assumptions \ref{assumption_referene}-\ref{assumption_disturbance}, under the funnel control law~(\ref{novel funnel controller}). Choose funnel boundaries $\varphi_i$ $(i=0,1,2,3)$ as in~\eqref{novel funnel} such that the initial values satisfy
    \[
        \varphi_i(0) |e_{i}(0)|<1\quad (i=0,1,2,3).
    \]
    Additionally, assume that the functions $\varphi_0, \varphi_1$ and $z_{h_{ref}}$ satisfy the following condition:
    \begin{equation}\label{inequality-step 1}
    	\begin{aligned}
    		&\exists\, \mu\in (0,1)\ \forall\, t\ge 0:\\
    		&\;\;\;\;\frac{{|{{\dot \varphi }_0}(t)|}}{{V\varphi _0^2(t)}} + \frac{1}{{{\varphi _1}(t)}} + \frac{{{D_0}}}{V} + \frac{{(1 + V)}}{V}|{{\dot z}_{{h_{ref}}}}(t)| \le \mu.
    	\end{aligned}
    \end{equation}
Then the funnel controller~(\ref{novel funnel controller}) applied to~(\ref{dzh})-(\ref{beta}) yields an
initial-value problem which has a solution, every solution can be maximally extended and every maximal solution
$(z_h,\psi_V,\psi,\omega_y):[0,t_f)\to\mathbb{R}^4$, $t_f\in(0,\infty]$ has the following properties:
	\begin{itemize}
        \item global existence: $t_f = \infty$;
		\item all errors evolve uniformly in the respective prescribed performance funnels, that is for all $i=0,1,2,3$ there exists ${\varepsilon _i} \in (0,1)$ such that for all $t\ge 0$ we have $|{\varpi _i}\left( t \right)| \le {\varepsilon _i}$.
		\item all signals $z_{h}$, $\psi_V$, $\psi$, $\omega_y$, $\delta_y$ and $k_i$ $(i=0,1,2,3)$ in the closed-loop system are bounded.
	\end{itemize}
\end{theorem}
\begin{proof}
Before the analysis, we record that it follows from \eqref{dy0}-\eqref{dy3} that the derivatives of $e_i(t)$ $(i=0,1,2,3)$ can be expressed as
\begin{align}\label{de0}
	{{\dot e}_0}\left( t \right) &=  - V\sin \left( {{y_1}\left( t \right)} \right) - {{\dot z}_{{h_{ref}}}}\left( t \right)+ {\Delta _0}(t) \\\nonumber\label{de1}
	{{\dot e}_1}\left( t \right) &= {e_2}\left( t \right) - {k_1}\left( t \right){\varpi _1}\left( t \right) + \frac{d}{{dt}}\left( {{k_0}\left( t \right){\varpi _0}\left( t \right)} \right)\\
	& \;\;\;\;-(1+c_2) {y_2}\left( t \right) - {c_1} + {\Delta _1}(t)\\\nonumber\label{de2}
	{{\dot e}_2}\left( t \right) &= {e_3}\left( t \right) - {k_2}\left( t \right){\varpi _2}\left( t \right) + \frac{d}{{dt}}\left( {{k_1}\left( t \right){\varpi _1}\left( t \right)} \right)\\
	&\;\;\;\; + {c_1} + {c_2}{y_2}\left( t \right) + {\Delta _2}(t) - {\Delta _1}(t)\\\label{de3}\nonumber
	{{\dot e}_3}\left( t \right) &= {c_5}u\left( t \right) - z_{{h_{ref}}}^{\left( 4 \right)}\left( t \right) + \frac{d}{{dt}}\left( {{k_2}\left( t \right){\varpi _2}\left( t \right)} \right)\\
	&\;\;\;\;+ {c_3} + {c_4}{y_2}\left( t \right) + {\Delta _3}(t)
\end{align}
In the following, we will first show that a local solution exists on $[0,t_f)$ and the tracking error evolves uniformly within the prescribed performance funnel, and we show $t_f=\infty$ in the last step.

\textcolor{black}{
\textbf{\textit{Step 1:}} To show existence of a solution of the closed-loop system, consider the functions
\begin{equation}
	\begin{aligned}
		{{\tilde e}_0}:{D_0} \to \mathbb{R},\ \left( {t,{y_0}} \right) \mapsto {y_0} - {z_{h_{ref}}}(t)
	\end{aligned}
\end{equation}
with the set ${D_0}: = {\mathbb{R}_{ \ge 0}} \times\mathbb{R}$ and
\begin{equation}
	\begin{aligned}
		&{{\tilde e}_i}:{D_i} \to \mathbb{R},\ \left( {t,{y_0}, \ldots ,{y_i}} \right) \mapsto {y_i} - z_{h_{ref}}^{(i)}(t) \\
		&\;\;\;\;\;\;\;\;\;\;\;\;\;+ \frac{{{\varphi _{i - 1}}(t){{\tilde e}_{i - 1}}(t,{y_0}, \ldots ,{y_{i - 1}})}}{{1 - \varphi _{i - 1}^2(t)\tilde e_{i - 1}^2(t,{y_0}, \ldots ,{y_{i - 1}})}}
	\end{aligned}
\end{equation}
with the sets
\begin{equation}
	\begin{aligned}
		&{D_i}: = \left\{ {\left( {t,{y_0}, \ldots ,{y_i}} \right) \in {D_{i - 1}} \times \mathbb{R}\mid} \right.\\
		&\;\;\;\;\;{\varphi _{i - 1}}(t)|{{\tilde e}_{i - 1}}\left( {t,{y_0}, \ldots ,{y_i}} \right)| < 1\} ,\ \left( {i = 1,2,3} \right),\\
		&{D_4}: = \left\{ {\left( {t,{y_0}, \ldots ,{y_3}} \right) \in {D_3}\mid} \right.\\
		&\;\;\;\;\;{\varphi _3}(t)|{{\tilde e}_3}\left( {t,{y_0}, \ldots ,{y_3}} \right)| < 1\}.
	\end{aligned}
\end{equation}
Introducing $Y(t) = (y_0(t), \ldots, y_3(t))^\top$ and the function
\begin{equation}
	\begin{aligned}
		&F:{D_4} \to {\mathbb{R}^4},\;(t,{y_0}, \ldots ,{y_3})\\
		&\mapsto \left( {\begin{array}{*{20}{c}}
				{ - V\sin ({y_1}) + {\Delta _0}}(t)\\
				{ - {c_1} - {c_2}{y_2} + {\Delta _1}}(t)\\
				{{y_3} + {c_1} + {c_2}{y_2} + {\Delta _2}(t) - {\Delta _1}}(t)\\
				{{c_3} + {c_4}{y_2} + {\Delta _3}(t) - {c_5}\frac{{{\varphi _3}(t){{\tilde e}_3}(t,{y_0}, \ldots ,{y_3})}}{{1 - \varphi _3^2(t)\tilde e{{_3^2}}(t,{y_0}, \ldots ,{y_3})}}}
		\end{array}} \right)
	\end{aligned}
\end{equation}
the closed-loop system takes the form
\begin{equation}
	\dot Y(t) = F(t,Y(t)),\quad Y(0) = (y_0^0, \ldots, y_3^0)^\top.
\end{equation}
Since $(0,Y(0))\in D_4$ and $F$ is measurable in~$t$, continuous in $(y_0,\ldots,y_3)$ and locally essentially bounded, an application of Theorem~B.1 from \cite{Ilchmann A 2009 International Journal of Control} yields the existence of a solution  and every solution can be extended to a maximal solution $Y:[0,t_f)\to\mathbb{R}^4$ with $t_f\in(0,\infty]$.} Furthermore, the graph of $Y$ is not a compact subset of~$D_4$.

\textbf{\textit{Step 2:}} We show that $k_0$ is bounded  on $[0,t_f)$. According to the definition of $\varpi_0(t)$ and (\ref{de0}) we have
\begin{equation}
\begin{aligned}\label{domega0}
{{\dot \varpi }_0}\left( t \right) &= \frac{{{{\dot \varphi }_0}\left( t \right)}}{{{\varphi _0}\left( t \right)}}{\varpi _0}\left( t \right) - {\varphi _0}\left( t \right)V\sin \left( {{y_1}\left( t \right)} \right)\\
& + {\varphi _0}\left( t \right)\left( {{\Delta _0}(t) - {{\dot z}_{{h_{ref}}}}\left( t \right)} \right).
\end{aligned}
\end{equation}
By the mean value theorem, for each $t\in [0,t_f)$, there exists $\xi \left( t \right)$ between $- {k_0}\left( t \right){\varpi _0}\left( t \right)$ and $- {k_0}\left( t \right){\varpi _0}\left( t \right) + {e_1}\left( t \right) + {{\dot z}_{{h_{ref}}}}\left( t \right)$ such that
\begin{equation}
\begin{aligned}\label{sin(y1)}
\sin \left( {{y_1}\left( t \right)} \right)&= \sin \left( { - {k_0}\left( t \right){\varpi _0}\left( t \right) + {e_1}\left( t \right) + {{\dot z}_{{h_{ref}}}}\left( t \right)} \right)\\
&= \sin \left( { - {k_0}\left( t \right){\varpi _0}\left( t \right)} \right) \\
&+ \left( {{e_1}\left( t \right) + {{\dot z}_{{h_{ref}}}}\left( t \right)} \right)\cos \left( {\xi \left( t \right)} \right).
\end{aligned}
\end{equation}
Now define ${U_0}(t) = \frac{1}{2}\varpi _0^2\left( t \right)$, then from (\ref{inequality-step 1}), (\ref{domega0}) and (\ref{sin(y1)}), and invoking $|{{\varpi _0}\left( t \right)}|<1$ and $|e_1(t)|<1/\varphi_1(t)$, we find that
\begin{equation}
\begin{aligned}
&{\dot U_0}(t) =  {\varpi _0}\left( t \right){{\dot \varpi }_0}\left( t \right)\\
&= \frac{{{{\dot \varphi }_0}(t)}}{{{\varphi _0}(t)}}{\varpi _0^2}{\left( t \right)} + {\varphi _0}\left( t \right){\varpi _0}\left( t \right)( - V\sin ( - {k_0}(t){\varpi _0}\left( t \right))\\
&\;\;\;\;- V({e_1}(t) + {{\dot z}_{{h_{ref}}}}(t))\cos (\xi (t)) + {\Delta _0}(t) - {{\dot z}_{{h_{ref}}}}(t))\\
&\le  - V{\varphi _0}(t)\sin \left( { - {k_0}(t){\varpi _0}\left( t \right)} \right){\varpi _0}\left( t \right) + \frac{{|{{\dot \varphi }_0}(t)|}}{{{\varphi _0}(t)}}\\
&\;\;\;\;+ V{\varphi _0}(t)\left( {\frac{1}{{{\varphi _1}(t)}} + \frac{{{D_0}}}{V} + \frac{{(1 + V)}}{V}|{{\dot z}_{{h_{ref}}}}(t)|} \right)\\
&\le V{\varphi _0}(t)\left( -N(-{\varpi _0}\left( t \right)) + \mu \right)
\end{aligned}
\end{equation}
where $N:(-1,1)\to(-1,1),\ s\mapsto \sin\left(\tfrac{s}{1-s^2}\right) s$. The function $N$ is symmetric and satisfies $N(-s) = N(s)$.

Choose $\varepsilon_0\in (0,1)$ such that $|\varpi_0(0)|<\varepsilon_0$ and $N(\varepsilon_0) < -\mu$. In the following, we show that $|\varpi_0(t)| \le \varepsilon_0$ for all $t\in [0,t_f)$. Assume there exists some $t\in [0,t_f)$ with $|\varpi_0(t)|>\varepsilon_0$ and define
\[
    {{\bar t}_0}: = \inf \left\{ {t \in [0,\textcolor{black}{t_f} )\mid \left| {{\varpi _0}(t)| > \varepsilon_0 } \right.} \right\} > 0.
\]
Since $N$ is continuous there exists $\eta>0$ such that $N(s) \le -\mu$ for all $s\in \mathbb{R}$ with $|s-\varepsilon_0| \le \eta$.
By symmetry of $N$ we also have $N(s) \le -\mu$ for all $s\in \mathbb{R}$ with $|s+\varepsilon_0| \le \eta$. Since $\varpi_0$ is continuous with $|\varpi_0(\bar t_0)| = \varepsilon_0$, there exists $\bar t_1\in (\bar t_0,t_f)$ such that $|\varpi_0(t)| > \varepsilon_0$  and $|\varpi_0(\bar t_0) - \varpi_0(t)| < \eta$ for all $t\in (\bar t_0,\bar t_1]$.

Let $\sigma = \mathop{\rm sgn} \varpi_0(\bar t_0)$, then $\varpi_0(\bar t_0)= \sigma |\varpi_0(\bar t_0)| = \sigma \varepsilon_0$ and hence $|\varpi_0(t) - \varpi_0(\bar t_0)| = |\varpi_0(t) - \sigma \varepsilon_0| \le \eta$ for all $t\in [\bar t_0,\bar t_1]$.  It follows that $N(\varpi_0(t))\le -\mu$ for all $t\in [\bar t_0,\bar t_1]$, and hence ${\dot U_0}(t) \le V\varphi_0(t) \big( \textcolor{black}{N(\varpi_0(t))} + \mu\big) \le 0$,  which upon integration gives $\varepsilon_0^2 = \varpi_0(\bar t_0)^2 \ge \varpi_0(\bar t_1)^2 > \varepsilon_0^2$, a contradiction. Hence,  $|\varpi_0(t)| \le \varepsilon_0$ for all $t\in [0,t_f)$ and thus $k_0$ is bounded \textcolor{black}{on} $[0,t_f)$.

\textbf{\textit{Step 3:}}  We show that $k_1$ is bounded  on $[0,t_f)$. A standard procedure in funnel control is used by seeking a contradiction.
For some $\varepsilon_1\in (0,1)$, which we will determine later, suppose that  there exists  $t_1^*\in [0,t_f)$ such that ${\varpi _1}\left(t_1^* \right) > \varepsilon_1$ and define
\[
    t_0^* := \max\{t\in [0,t_1^*) \mid {\varpi _1}\left(t \right) = \varepsilon_1\}.
\]
Then we find that
\begin{align}\label{er/2}
\forall\, t \in \left[ {t_0^*,t_1^*}\right]:\ {\varpi _1}\left( t \right) \ge {\varepsilon _1}
\end{align}
and hence
\begin{equation}\label{k1>}
\begin{aligned}
\forall\, t \in \left[ {t_0^*,t_1^*} \right]:\ {k_1}\left( t \right) = \frac{1}{{1 - \varpi _1^2\left( t \right)}} \ge \frac{1}{{1 - \varepsilon _1^2}}.
\end{aligned}
\end{equation}
Define ${U_1}(t) = \frac{1}{2}\varpi _1^2\left( t \right)$, then according to \eqref{de1} we have
\begin{equation}
\begin{aligned}\label{dU1}
&{{\dot U}_1}(t)= {\varpi _1}\left( t \right){{\dot \varpi }_1}\left( t \right)\\
&= {\varpi _1}\left( t \right)\left( {{{\dot \varphi }_1}\left( t \right){e_1}\left( t \right) + {\varphi _1}\left( t \right){{\dot e}_1}\left( t \right)} \right)\\
&= \frac{{{{\dot \varphi }_1}(t)}}{{{\varphi _1}(t)}}{\varpi _1^2}{(t)} - {\varphi _1}(t){k_1}(t){\varpi _1^2}{(t)} - {c_1}{\varphi _1}(t){\varpi _1}(t)\\
&\;\;\;\;+ {\varphi _1}(t){\varpi _1}(t)\left( { - {c_2}{y_2}(t)+{\Delta _1}(t)+ {e_2}(t) - {y_2}(t)} \right)\\
&\textcolor{black}{\;\;\;\;+ {\varphi _1}(t){\varpi _1}(t)\left( {1 + 2{\varpi _0^2}{{(t)}}{k_0}\left( t \right)} \right){k_0}(t){{\dot \varpi }_0}(t)}.
\end{aligned}
\end{equation}
From Step 2 we have that $k_0$ is bounded on $t\in [0,t_f)$. Furthermore, $\dot \varpi_0$ is bounded by \eqref{domega0}, $|e_2(t)|<1/\varphi_2(t)$ and $y_2, \Delta_1, \varphi_1, \dot \varphi_1$ are clearly bounded as well, \textcolor{black}{hence} there exists an upper bound $C_1>0$ such that for all
$t\in [0,t_f)$ we have
\begin{equation}
\begin{aligned}\label{c1}
&\ \frac{|\dot \varphi_1(t)|}{\varphi_1^2(t)} +  c_1 + (1+c_2) |y_2(t)| + |\Delta_1(t)| + |e_2(t)|\\
& \textcolor{black}{ + \left( {1 + 2{\varpi _0^2}{{(t)}}{k_0}\left( t \right)} \right){k_0}(t){{\dot \varpi }_0}(t) \le C_1}.
\end{aligned}
\end{equation}
Invoking $|{{\varpi _1}\left( t \right)}|<1$ we thus obtain
\begin{equation}
\begin{aligned}
{{\dot U}_1}(t) &\le {\varphi _1}(t)|{\varpi _1}(t)|\left( { - {k_1}(t)|{\varpi _1}(t)| + C_1} \right) \textcolor{black}{\le  0}
\end{aligned}
\end{equation}
when we choose $\varepsilon_1 \in (0,1)$ large enough so that $\frac{\varepsilon_1}{1-\varepsilon_1^2} \ge C_1$.
Upon integration over $[t_0^*, t_1^*]$ we find that
\begin{equation}
\begin{aligned}
\varepsilon_1^2 = \omega_1(t_0^*)^2 \ge  \omega_1(t_1^*)^2 > \varepsilon_1^2,
\end{aligned}
\end{equation}
a contradiction. Hence  $|\varpi_1(t)| \le \varepsilon_1$ for all $t\in [0,t_f)$, and thus $k_1$ is bounded on $[0,t_f)$.

\textbf{\textit{Step 4:}} We prove that $k_2$ is bounded  on $[0,t_f)$. With ${U_2}(t) = \frac{1}{2}\varpi _2^2\left( t \right)$ it follows from \eqref{de2} that
\begin{equation}
\begin{aligned}
&{{\dot U}_2}(t) = {\varpi _2}\left( t \right){{\dot \varpi }_2}\left( t \right)\\
&= {\varpi _2}\left( t \right)\left( {{{\dot \varphi }_2}\left( t \right){e_2}\left( t \right) + {\varphi _2}\left( t \right){{\dot e}_2}\left( t \right)} \right)\\
&= \frac{{{{\dot \varphi }_2}(t)}}{{{\varphi _2}(t)}}{\varpi _2^2}{(t)} - {\varphi _2}(t){k_2}(t){\varpi _2^2}{(t)}\\
& \;\;\;\;+ {\varphi _2}(t){\varpi _2}(t)\left( {{c_1} + {c_2}{y_2}(t) + {\Delta _2}(t) - {\Delta _1}(t) + {e_3}(t)} \right)\\
&\textcolor{black}{\;\;\;\;+ {\varphi _2}(t){\varpi _2}(t)\left( {1 + 2{\varpi _1^2}{{(t)}}{k_1}\left( t \right)} \right){k_1}(t){{\dot \varpi }_1}(t)}.
\end{aligned}
\end{equation}
By Step 3 it follows that $k_1$ is bounded on $t\in [0,t_f)$. Furthermore, $|e_3(t)|<1/\varphi_3(t)$ and $\varpi_1$, $\dot \varpi_1$, $y_2, \Delta_1, \Delta_2, \varphi_2, \dot \varphi_2$ are  bounded on $[0,t_f)$, hence there exists a constant $C_2>0$ satisfying
\textcolor{black}{
\begin{equation}
\begin{aligned}
&\frac{{|{{\dot \varphi }_2}(t)|}}{{{\varphi _2^2}(t)}} + {c_1} + {c_2}|{y_2}(t)| + |{\Delta _2}(t)| + |{\Delta _1}(t)| + |{e_3}(t)| \\
&\textcolor{black}{+ \left( {1 + 2{\varpi _1^2}{{(t)}}{k_1}\left( t \right)} \right){k_1}(t)|{{\dot \varpi }_1}(t)| \le {C_2}}.
\end{aligned}
\end{equation}
Invoking $|{{\varpi _2}\left( t \right)}|<1$ we thus obtain
\begin{equation}
\begin{aligned}
{{\dot U}_2}(t) &\le {\varphi _2}(t)|{\varpi _2}(t)|\left( { - {k_2}(t)|{\varpi _2}(t)| + C_2} \right)
\end{aligned}
\end{equation}
and with a similar argument as in Step 3 it can be shown that $k_2$ is bounded on ${[0,t_f)} $.}

\textbf{\textit{Step 5:}} We show that $k_3$ is bounded  on $[0,t_f)$. To this end, we substitute \eqref{novel funnel controller} into \eqref{de3}, yielding
\begin{equation}
\begin{aligned}
{{\dot e}_3}\left( t \right) &= - {c_5}{k_3}(t)\frac{{{\varpi _3}(t)}}{{{\textcolor{black}{\varphi _3}\left( t \right)}}} - z_{{h_{ref}}}^{\left( 4 \right)}\left( t \right) + \frac{d}{{dt}}\left( {{k_2}\left( t \right){\varpi _2}\left( t \right)} \right)\\
&\;\;\;\;+ {c_3} + {c_4}{y_2}\left( t \right) + {\Delta _3}(t)
\end{aligned}
\end{equation}
Define ${U_3}(t) = \frac{1}{2}\varpi _3^2\left( t \right)$ and calculate
\begin{equation}
\begin{aligned}
&{{\dot U}_3}(t) = {\varpi _3}\left( t \right){{\dot \varpi }_3}\left( t \right)\\
&= {\varpi _3}\left( t \right)\left( {{{\dot \varphi }_3}\left( t \right){e_3}\left( t \right) + {\varphi _3}\left( t \right){{\dot e}_3}\left( t \right)} \right)\\
&=  - {c_5}{k_3}(t)\textcolor{black}{\varpi _3^2(t)} + \frac{{{{\dot \varphi }_3}(t)}}{{{\varphi _3}(t)}}\varpi _3^2(t)\\
&\;\;\;\;+ {\varpi _3}\left( t \right){\varphi _3}\left( t \right)\left( { - z_{{h_{ref}}}^{\left( 4 \right)}\left( t \right) + {c_3} + {c_4}{y_2}\left( t \right) + {\Delta _3}(t)} \right)\\
&\textcolor{black}{\;\;\;\;+ {\varpi _3}\left( t \right){\varphi _3}\left( t \right)\left( {1 + 2\varpi _2^2(t){k_2}\left( t \right)} \right){k_2}(t){{\dot \varpi }_2}(t)}.
\end{aligned}
\end{equation}
From Step 4 we find that $k_2$ is bounded on $[0,t_f)$. Because of $|e_3(t)|<1/\varphi_3(t)$ and boundedness of $\varpi_2$, $\dot \varpi_2$, $y_2, \Delta_3, \varphi_3, \dot \varphi_3,\textcolor{black}{z_{{h_{ref}}}^{\left( 4 \right)}}$ on $[0,t_f)$ it follows that there exists $C_3>0$ such that
\begin{equation}
\begin{aligned}
&\textcolor{black}{{{\varphi _3}(t)}\Big(\frac{{|{{\dot \varphi }_3}(t)|}}{{{\varphi _3^2}(t)}} + \left|z_{{h_{ref}}}^{\left( 4 \right)}\left( t \right)\right| + {c_3} + {c_4}|{y_2}\left( t \right)|}  \\
&\textcolor{black}{+ |{\Delta _3}| + \left( {1 + 2\varpi _2^2(t){k_2}\left( t \right)} \right){k_2}(t)|{{\dot \varpi }_2}(t)|\Big)\le C_3}.
\end{aligned}
\end{equation}
Invoking $|{{\varpi _3}\left( t \right)}|<1$ we thus obtain
\begin{equation}
\begin{aligned}
\textcolor{black}{{{\dot U}_3}(t) \le |{\varpi _3}(t)|\left( { - {c_5}{{{k_3}(t)|{\varpi _3}(t)|}} + {C_3}} \right)}
\end{aligned}
\end{equation}
and with a similar argument as in Step 3 it can be shown that $k_3$ is bounded on ${[0,t_f)} $.
	
\textbf{\textit{Step 6:}} \textcolor{black}{We show that $t_f=\infty$.} \textcolor{black}{Assuming $t_f<\infty$ it follows from Steps 2--5 that the closure of the graph of $(y_0,\ldots,y_3)$ is a compact subset of $D_4$, which contradicts the findings of Step~1. Therefore, $t_f=\infty$.}
\end{proof}

Note that the work \cite{XiyuGu 2022 ISA} also employs funnel control techniques to solve the tracking problem with prescribed transient behavior for RVs. However, the controller there requires additional design parameters which need to be sufficiently large, but it is not known a priori how large they must be chosen. In the present paper, we avoid this problem by introducing a novel error variable form as in~\eqref{e_i}. This seems advantageous for practical engineering.

\section{\textcolor{black}{SIMULATION}\label{section simulation}}

In this section, we illustrate the performance of the funnel controller~\eqref{novel funnel controller} by considering the  lateral action of a RV with constant speed $V=5Ma$ at a height of \textcolor{black}{20 km}.
The initial states of the RV and the values of the geometric system parameters are shown in \textcolor{black}{Table \ref{initial value}}. As disturbances we choose ${\Delta _0}(t) = \frac{1}{{57.3}}5\sin \left( {\frac{\pi }{4}t} \right)$, ${\Delta _1}(t) = \frac{1}{{57.3}}0.2\sin \left( {\frac{\pi }{4}t} \right)$, ${\Delta _2}(t) = \frac{1}{{57.3}}2\sin \left( {\frac{\pi }{4}t} \right)$, ${\Delta _3}(t) = \frac{1}{{57.3}}10\sin \left( {\frac{\pi }{4}t} \right)$. The simulation was performed in MATLAB (solver: \textsc{ode45}, default tolerances).

In practical engineering applications, input constraints are always present. Therefore, although such constraints are not considered in the theoretical treatment in Theorem~\ref{theorem}, we incorporated them in the simulation such that the actual control input is $\sat(u(t))$, where $\sat(v) = v$ for $|v|\le 40$ and $\sat(v) = \sgn(v) 40$ for $|v|>40$.

As for the maneuvering reference trajectory, an extensively used Dubins trajectory is selected as the planning path, shown in \textcolor{black}{Fig. \ref{fig_dubin_trajectory}}, where $\theta_j$ $(j=BC,DE,FG)$ and $R_j$ are the turning radius and central angles, respectively. The coordinates of the starting point, turning points and end point are $A\left( {0,450} \right)$, $B\left( {12000,50} \right)$, $C\left( {15000,0} \right)$, $D\left( {24000,0} \right)$, $E\left( {27000,50} \right)$, $F\left( {30000,150} \right)$, $G\left( {35000,260} \right)$ and \textcolor{black}{$H\left( {48530,-200} \right)$}.
The time consumed per period is calculated by ${t_i} = \frac{{{{\bar l}_i}}}{V}$ $\left( {i = AB,CD,EF,GH} \right)$ through straight regions and ${t_j} = \frac{{{{\mathord{\buildrel{\lower3pt\hbox{$\scriptscriptstyle\frown$}} \over l} }_j}}}{V} = \frac{{{\theta _j}{R_j}}}{V}$ $\left( {j = BC,DE,FG} \right)$ in turning areas.
Since the proposed performance funnel is time-triggered based on the planning path, the triggered times are set in accordance with the time points of the reference trajectory, resulting in $p=3$.  Thus the parameters of the funnel boundary function of the form~\eqref{novel funnel} are chosen as in \textcolor{black}{Table~\ref{tabel funnel}}.

\begin{figure}
	\centering
	\includegraphics[width =1\columnwidth]{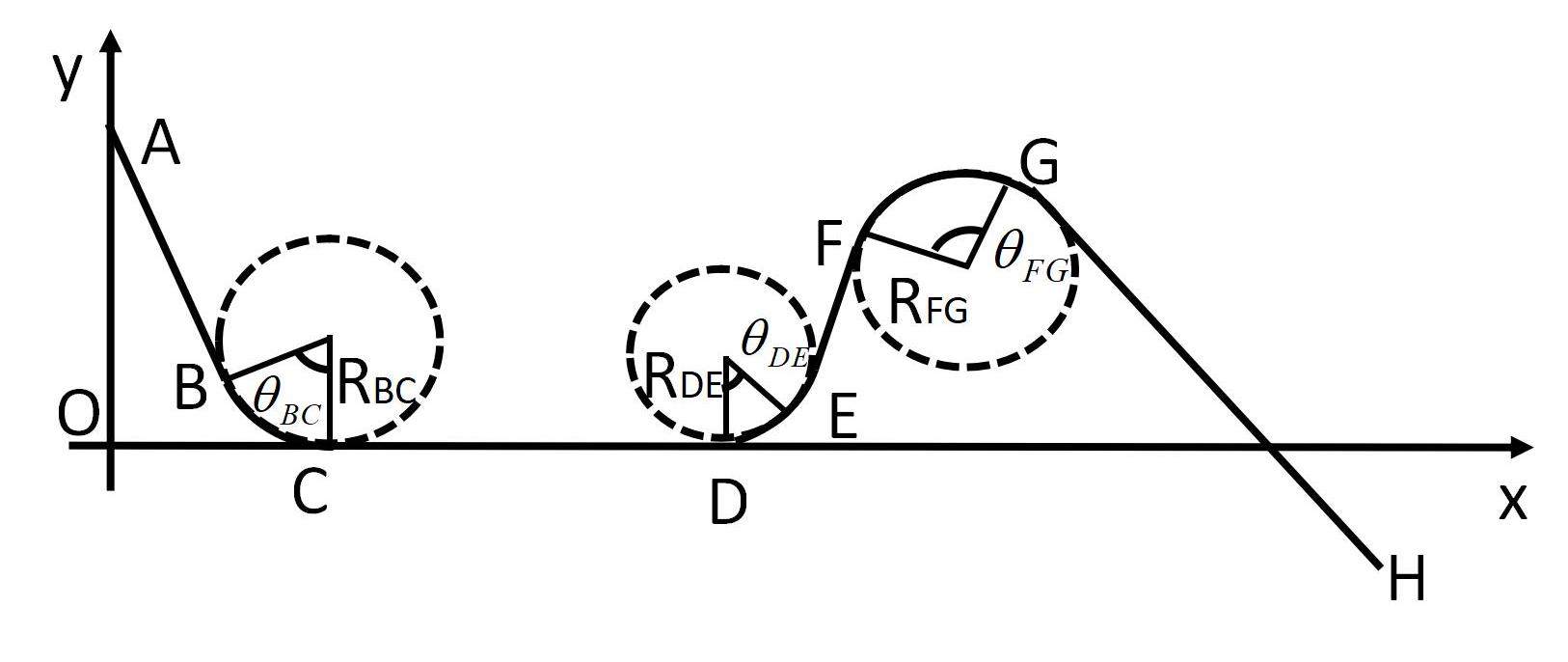}
	\caption{\textcolor{black}{Planning trajectory for RV}}
	\label{fig_dubin_trajectory}
\end{figure}

The simulation is depicted in Figs.~\ref{fig_e0}--\ref{fig_/zh}.  The tracking error $e_0$ and the auxiliary errors $e_1, e_2, e_3$ are shown in Figs.~\ref{fig_e0}--\ref{fig_e3}. It is found that $e_3$ is more sensitive to maneuvering than $e_0$, where a slight jump appears near each trigger time. Every error $e_i$ $(i=0,1,2,3)$ is kept within its respective funnel through the whole process and under the influence of disturbances. The control input, namely the rudder angle, is shown in Fig.~\ref{fig_deltay}. It can be seen that the control input reaches saturation due to the input constraints. Although the saturation is not covered by Theorem~\ref{theorem}, the controller obviously exhibits an exceptional performance under input constraints. Further research is necessary to find theoretical guarantees for this performance.

Fig.~\ref{fig_/psi_psiv_beta_wy} shows the deflection angle, sideslip angle, yaw angle and  yaw rate to illustrate that all variables in the closed-loop system are bounded. The tracking maneuver reference trajectory is shown in Fig.~\ref{fig_/zh} together with the output signal generated under control. Although the performance of the vehicle in the continuous large maneuver segment is not as good as that in the single maneuver segment, we observe a decent tracking performance overall.

However, we like to note that in this example the theoretical condition~\eqref{inequality-step 1} from Theorem~\ref{theorem} is not satisfied, yet the controller still works. This shows that the assumptions of Theorem~\ref{theorem} are quite conservative and further research is necessary to relax them. A thorough inspection of the proof of Theorem~\ref{theorem} reveals that the conservativeness of condition~\eqref{inequality-step 1} is due to the utilization of the mean value theorem and avoiding it could lead to a weaker condition.

\begin{figure}
	\centering
	\includegraphics[width =1\columnwidth]{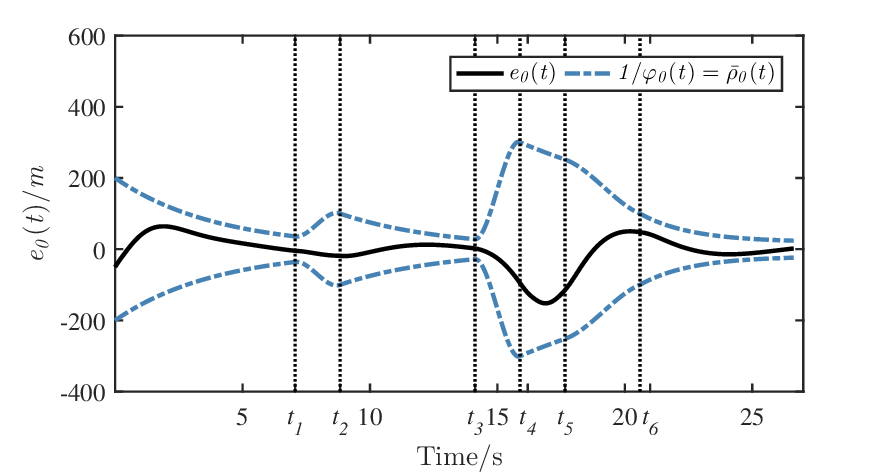}
	\caption{\textcolor{black}{Response of tracking error $e_0(t)$.}}
	\label{fig_e0}
\end{figure}
\begin{figure}
	\centering
	\includegraphics[width =1\columnwidth]{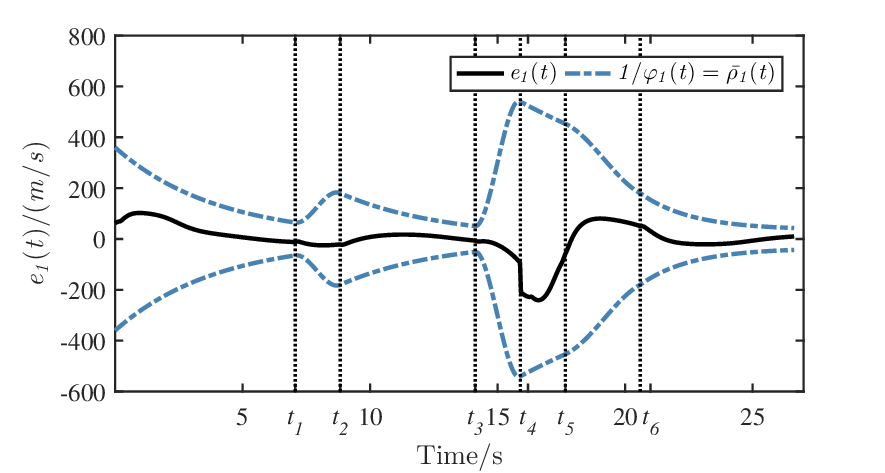}
	\caption{\textcolor{black}{Response of error $e_1(t)$.}}
	\label{fig_e1}
\end{figure}
\begin{figure}
	\centering
	\includegraphics[width =1\columnwidth]{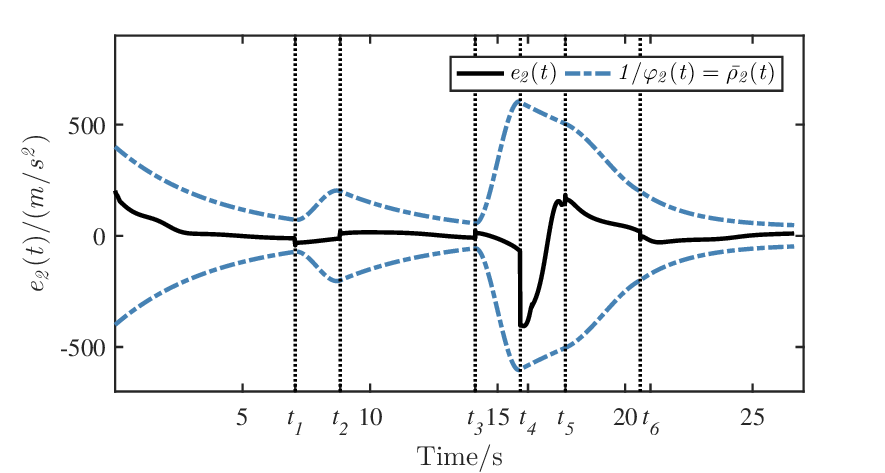}
	\caption{\textcolor{black}{Response of error $e_2(t)$.}}
	\label{fig_e2}
\end{figure}
\begin{figure}
	\centering
	\includegraphics[width =1\columnwidth]{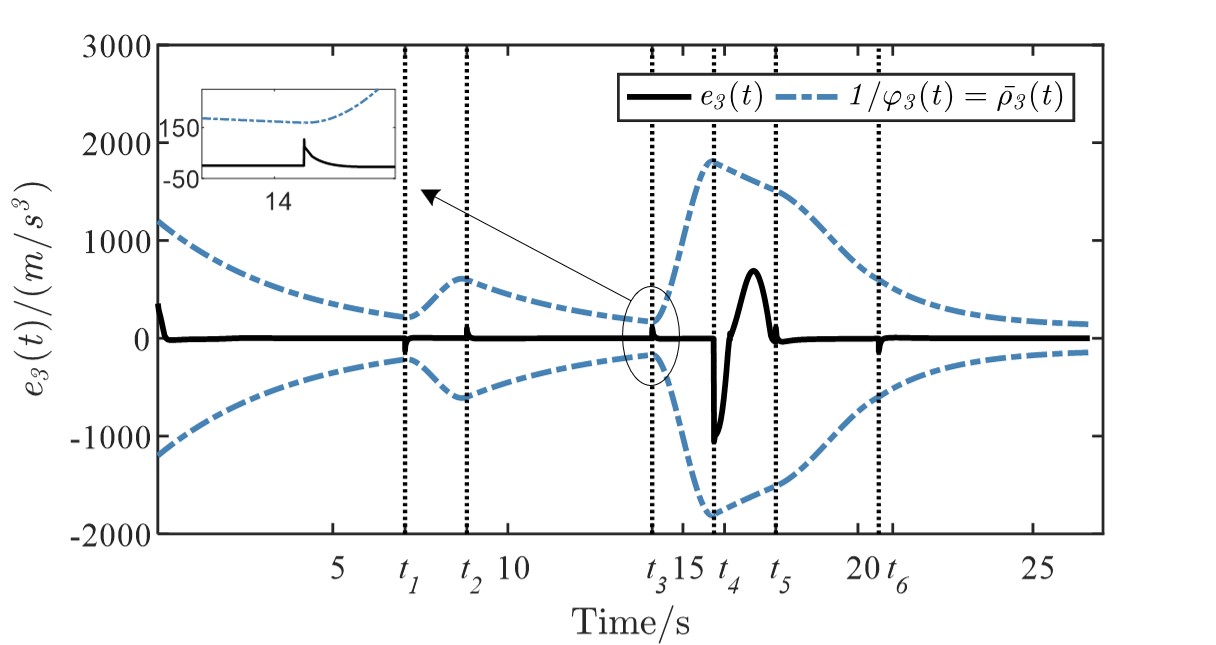}
	\caption{\textcolor{black}{Response of error $e_3(t)$.}}
	\label{fig_e3}
\end{figure}

\begin{figure}
	\centering
	\includegraphics[width =1\columnwidth]{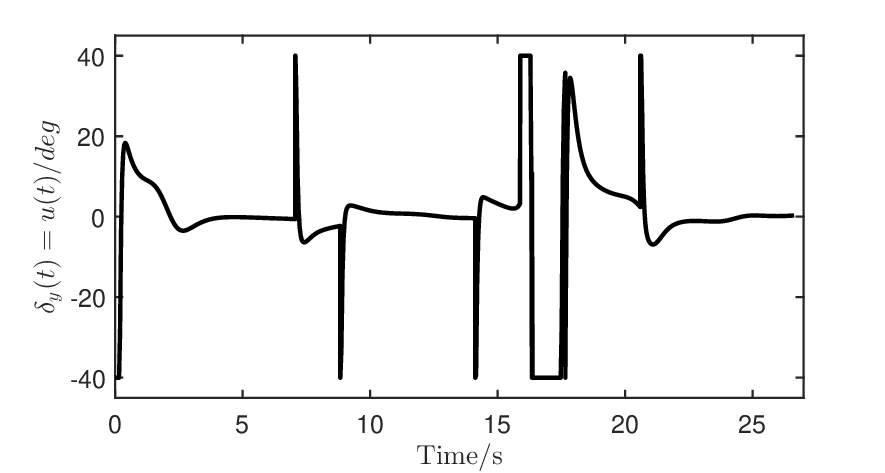}
	\caption{\textcolor{black}{Response of control input $\delta_y(t)=u(t)$.}}
	\label{fig_deltay}
\end{figure}

\begin{figure}
	\centering
	\includegraphics[width =1\columnwidth]{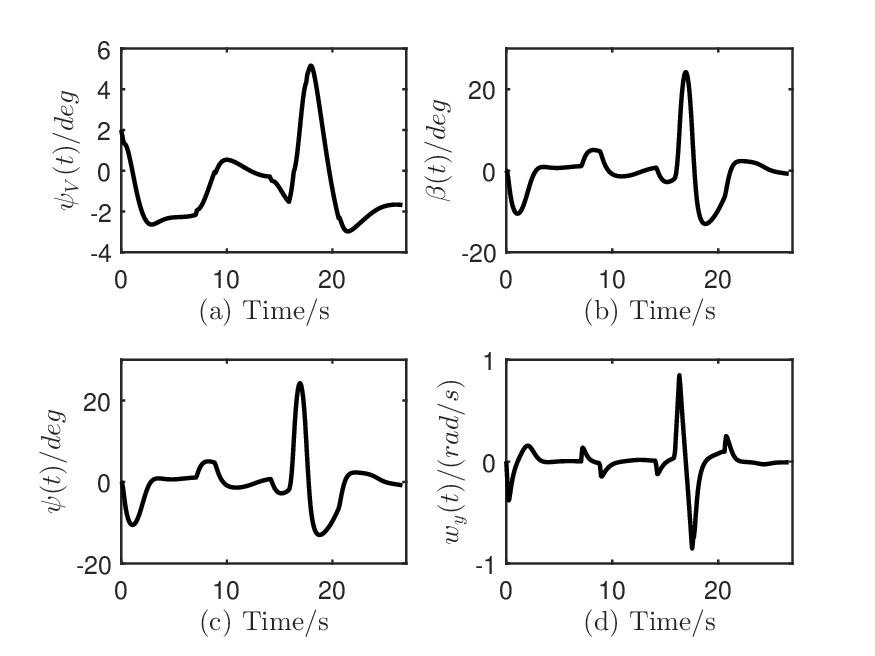}
	\caption{\textcolor{black}{Response of system: (a) $\psi_V(t)$, (b) $\beta(t)$, (c) $\psi(t)$, (d) $\omega_y(t)$.}}
	\label{fig_/psi_psiv_beta_wy}
\end{figure}

\begin{figure}
		\centering
		\includegraphics[width =1\columnwidth]{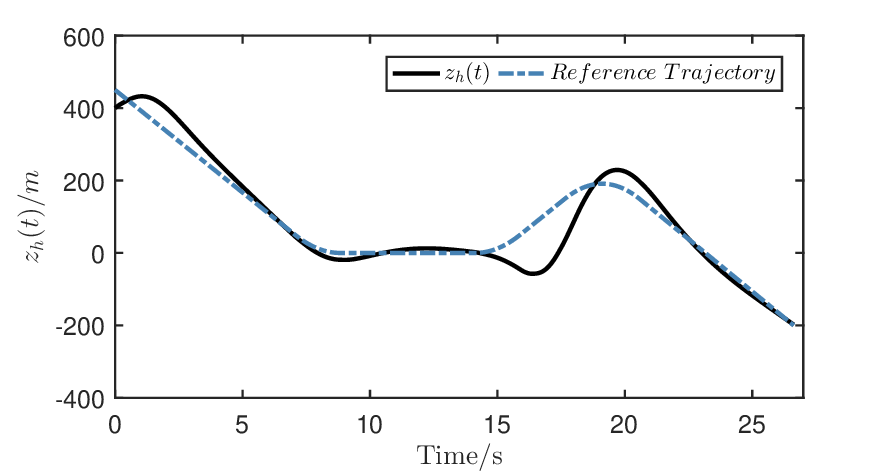}
		\caption{\textcolor{black}{Tracking response of flight altitude $z_h$.}}
		\label{fig_/zh}
\end{figure}

\begin{table}
	\centering
	\caption{\textcolor{black}{Geometric parameters and initial state of RV}}
	\label{initial value}
	\begin{tabular}{m{1.5cm}m{1.35cm}m{1.6cm}m{1.4cm}}
		\hline
		Variables&Value&Variables&Value\\
		\hline
		m (kg)&1200&$z_{h}(0)$ (m)&400\\
		S ($\rm{m^2}$)&1.3&$\psi_V(0)$ (rad) &2/57.3\\
		$l$ (m)&1.7&$\psi(0)$ (rad) &4/57.3\\
		$J_y$ $({\rm{kg}}\cdot{{\rm{m}}^2})$&8110&$\omega_y(0)$ (rad/s) &0.035\\
		$\alpha$ (rad)& 5/57.3&$\beta(0)$&2/57.3\\
        $c_z^\alpha$&0&$c_z^\beta$&0.1852\\
		$c_z^0$&-0.018714&$c_M^\alpha$&-0.1\\
		$c_M^\beta$&2.1335&$c_M^{\delta_y}$&5.1588\\
		$c_M^0$&0.18979&$\bar q$&3711.93329\\
		\hline
	\end{tabular}
\end{table}

\begin{table}
	\centering
	\caption{\textcolor{black}{Parameters of the proposed funnel boundary }}
	\label{tabel funnel}
	\begin{tabular}{m{2cm}m{4.8cm}}
		\hline
		Variables& ~~~~~~~~~~~~~~~~~~Value\\
		\hline
		$\bar \rho_0(t)$&$
		\bar \rho _{0,1}^0 = {\rm{200}}$, $\bar \rho _{0,2}^0 = 100$, $\bar \rho _{0,3}^0 = 300$, $\bar \rho _{0,4}^0 = 100$, $\bar \rho _{0,1}^\infty  = 2$, $\bar \rho _{0,2}^\infty  = 2$, $\bar \rho _{0,3}^\infty  = 2$, $\bar \rho _{0,4}^\infty  = 10$, ${l_{0,1}} = 0.25$, ${l_{0,2}} = 0.25$, ${l_{0,3}} = 0.5,{l_{0,4}} = 0.1$\\
		\hline
		$\bar \rho_i(t)$  $(i=1,2,3)$&${{\bar \rho }_{1}}\left( t \right) = 1.8{{\bar \rho }_{0}}\left( t \right)$, ${{\bar \rho }_{2}}\left( t \right) = 2{{\bar \rho }_{0}}\left( t \right)$, ${{\bar \rho }_{3}}\left( t \right) = 6{{\bar \rho }_{0}}\left( t \right)$\\
		\hline
	\end{tabular}
\end{table}

\section{\textcolor{black}{CONCLUSION}}\label{section conclusion}
In this paper, we proposed a novel funnel-based tracking control algorithm for guaranteeing prescribed performance of the tracking error in reentry vehicle maneuvering flight.
A time triggered non-monotonic funnel boundary is designed to successfully improve the controller performance, which is verified by simulations.
The results show that the proposed control method has the ability to stabilize the closed-loop system under arbitrary bounded disturbances.
Future research will focus on the relaxation of the conditions of Theorem~\ref{theorem} as well as their extension to the presence of input constraints. To this end, the recent results in~\cite{Berg22} might be a starting point.

\end{document}